\documentclass[12pt]{amsart}
\usepackage[top=1.5in, bottom=1.5in, left=1.4in, right=1.4in]{geometry}  
\usepackage{geometry}                
\usepackage{graphicx}
\usepackage{amssymb}
\usepackage{epstopdf}
\usepackage{color}
\usepackage{url}
\usepackage{hyperref}
\usepackage{verbatim}
\usepackage{tkz-graph}
\usepackage{enumitem}

\usepackage{xcolor}

\definecolor{darkbrown}{rgb}{.5,.1,.1} 

\title[Critical Groups of Adinkras]{The Critical Groups of Adinkras Up to $2$-Rank of Cayley Graphs}

\author{Chi Ho Yuen}
\address{Chi Ho Yuen:
Department of Mathematics, University of Oslo, Oslo, Norway
}
\email{chihy@math.uio.no}

\def\A{\mathcal{A}}

\def\Ccal{\mathcal{C}}

\def\K{\mathcal{K}}

\newcommand{\FF}{\mathbb{F}}

\newcommand{\ZZ}{\mathbb{Z}}

\numberwithin{equation}{section}
\theoremstyle{definition}

\newtheorem{theorem}{Theorem}[section]

\newtheorem{definition}[theorem]{Definition}
\newtheorem{corollary}[theorem]{Corollary}

\newtheorem{proposition}[theorem]{Proposition}
\newtheorem{example}[theorem]{Example}

\newcommand{\coker}{\operatorname{coker}}

\tikzset{VertexStyle/.style = {
    circle, fill=white, draw=black,
    text           = black,
    inner sep      = 2pt,
    outer sep      = 0pt,
    minimum size   = 10 pt}}

\tikzset{BoxVertex/.style = {rectangle, fill=white, draw=black, text = black, inner sep = 2.7pt, outer sep = 0pt, minimum size = 10pt}}

\tikzset{EdgeStyle/.style = {color=black!60, thick}}

\begin{document}

\begin{abstract}
Adinkras are graphical gadgets introduced by physicists to study supersymmetry, which can be thought of as the Cayley graphs for supersymmetry algebras.
Improving the result of Iga et al., we determine the critical group of an Adinkra given the $2$-rank of the Laplacian of the underlying Cayley graph.
As a corollary, we show that the critical group is independent of the signature of the Adinkra.
The proof uses the monodromy pairing on these critical groups.
\end{abstract}

\maketitle

\vspace{-5mm}
\section{Introduction}

Computing the cokernel, or equivalently, the Smith Normal Form (SNF), of an integer matrix is a recurrent topic in algebraic combinatorics \cite{StanleyNormal}.
While the definition is elementary, elaborated machinery has been used to approach the task, and the answer to many seemingly simple matrices remain unknown.
A prominent notion within the topic is the {\em critical group} $\K(G)$ of a graph $G$, which is (the torsion part of) the cokernel of the graph Laplacian $L(G)$ \cite{SandpileBook, ChipBook}.
A more specific direction is to compute the critical groups of graphs coming from algebra, where ideas from representation theory, Gr\"{o}bner theory, and $p$-adic number theory have been applied \cite{BKR_ChipFiring, CSX_Paley, DJ_Cayley, GMMY}.

In this paper, we study the critical groups of {\em Adinkras}, which are decorated (colored, dashed) graphs introduced by physicists to encode special supersymmetry algebras \cite{rA}.
Roughly speaking, the vertices are particles, and the edges of different colors correspond to the actions of different basis elements of the algebra acting on these particles, in which an edge is dashed if the action flips the sign of the particle.
Therefore, Adinkras can be thought of as a signed analogue of Cayley graphs for supersymmetry algebras.
In \cite{IKKY}, the authors initiated the problem, and proved the following using a ``deform to $\ZZ[x]$-modules'' argument:

\begin{theorem} \cite{IKKY} \label{thm:IKKY}
Let $\A$ be an $N$-colored Adinkra on $v$ vertices.
Then the odd component of $\K(\A)$ is isomorphic to that of $(\ZZ/(N^2-N)\ZZ)^{v/2}$.
\end{theorem}

However, the $2$-Sylow subgroup of $\K(\A)$ remains open, which seems tricky for its connection with another problem.
The underlying graph of an Adinkra is a Cayley graph $G$ of $\FF_2^n$, and the number of even invariant factors of $\K(\A)$ equals the corank of $L(G)$ over $\FF_2$, but this quantity is poorly-understood itself: even the simplest case of hypercube was a conjecture of Reiner before being solved by Bai \cite{Bai_hypercube}, and no good general description is known despite some recent progress \cite{GMMY}.
Nevertheless, we are able to bypass the apparent first hurdle and prove:

\begin{theorem} \label{mainthm}
Let $\A$ be an $N$-colored Adinkra on $v$ vertices.
Let the $2$-rank of the Laplacian of the underlying Cayley graph be $v/2-m$; we have $m\geq 0$ by \cite{GMMY}.
Then $\K(\A)\cong(\ZZ/2\ZZ)^m\oplus(\ZZ/\frac{N^2-N}{2}\ZZ)^m\oplus(\ZZ/(N^2-N)\ZZ)^{v/2-m}$.
\end{theorem}

Our proof differs from the main strategy in \cite{IKKY}, and uses {\em monodromy pairing} \cite{MonoPairing}.
While the pairing is defined for any critical group, the regularity of Adinkras yields a simple description of it \cite{DBS}.
We find a large ``orthonormal'' subset of $\K(\A)$, which implies that $\K(\A)$ contains a large subgroup whose structure is simple. Then we argue that the remaining part only depends on the $2$-rank.
To the best of our knowledge, this is the first time the monodromy pairing is being used to study the structure of specific critical groups.

As a corollary, we answer \cite[Conjecture~1]{IKKY} in the affirmative:

\begin{corollary} \label{maincoro}
$\K(\A)$ is independent of the signature of the Adinkra.
\end{corollary}

It is worth noting that while only the signed (dashed) graph structure of $\A$ was used to define $\K(\A)$, the existence of a compatible edge coloring is essential.
It is another curious instance in mathematics that admitting extra structure imposes constraints on seemingly irrelevant invariants of the object.

\section{Preliminaries} \label{sec:prelim}

Unless otherwise specified, all (signed) graphs are finite, simple, and connected.

\subsection{Adinkras}

\begin{definition}
A {\em signed graph} $G_{\sigma}$ is a graph $G$ together with an assignment $\sigma:E(G)\rightarrow\{\pm\}$ (a {\em signature} of $G$).
{\em Switching} a vertex flips the signs of the edges incident to it.
\end{definition}

\begin{definition} \label{def:Adinkra}
For $N\geq 2$, an $N$-colored {\em Adinkra} $\A$ is a signed graph with each edge colored by one of $N$ colors\footnote{Technically, the data here only defines a {\em Cliffordinkra}, as we need to further choose an acyclic orientation of the underlying graph to define an Adinkra. However, there are no compatibility conditions between the choice and the rest of the data, so we omit the orientation here.}, satisfying the following conditions:
\begin{enumerate}
\item the graph is bipartite;
\item every vertex is incident to exactly one edge of each color;
\item for every pair of distinct colors, the graph restricted to edges of these colors is a disjoint union of $4$-cycles;
\item each bi-colored $4$-cycle contains an odd number of negative edges.
\end{enumerate}
\end{definition}

\begin{example}

In Figure~\ref{fig:Adinkras}, on the left is a $3$-colored Adinkra supported on $Q_3$, and on the right is a $4$-colored Adinkra supported on $K_{4,4}$.
Negative edges are represented by dashed edges.

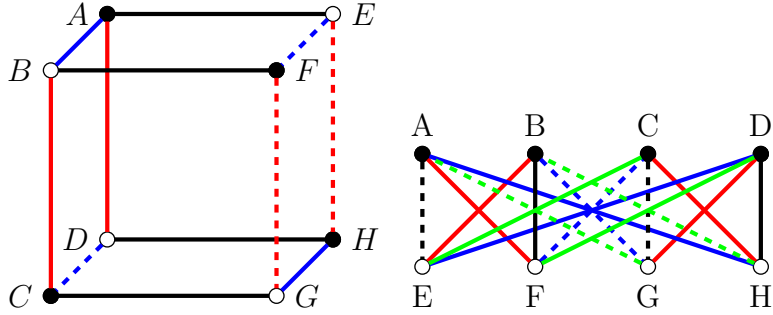
\begin{figure}
\begin{center}
\begin{tikzpicture}[scale=0.15]
\GraphInit[vstyle=Welsh]
\SetVertexNormal[MinSize=5pt]
\SetUpEdge[labelstyle={draw},style={ultra thick}]
\tikzset{Dash/.style={dashed,draw,ultra thick}}
\Vertex[x=0,y=0,Math,L={B},Lpos=180]{B}
\Vertex[x=20,y=0,Math,L={F},Lpos=0]{F}
\Vertex[x=0,y=-20,Math,L={C},Lpos=180]{C}
\Vertex[x=20,y=-20,Math,L={G},Lpos=0]{G}
\Vertex[x=5,y=5,Math,L={A},Lpos=180]{A}
\Vertex[x=25,y=5,Math,L={E},Lpos=0]{E}
\Vertex[x=5,y=-15,Math,L={D},Lpos=180]{D}
\Vertex[x=25,y=-15,Math,L={H},Lpos=0]{H}
\AddVertexColor{black}{F,C,A,H}
\Edge(A)(E)
\Edge(D)(H)
\Edge[color=red](A)(D)
\Edge[color=red,style=Dash](E)(H)
\Edge[color=blue](A)(B)
\Edge[color=blue,style=Dash](E)(F)
\Edge[color=blue,style=Dash](D)(C)
\Edge[color=blue](G)(H)
\Edge(B)(F)
\Edge(C)(G)
\Edge[color=red](B)(C)
\Edge[color=red,style=Dash](F)(G)
\end{tikzpicture}
\begin{tikzpicture}[scale=0.15]
\GraphInit[vstyle=Welsh] 
\SetVertexNormal[MinSize=5pt]  
\SetUpEdge[labelstyle={draw},style={ultra thick}]
\tikzset{Dash/.style = {ultra thick,dashed,draw}}
\Vertex[x=-10,y=0,Lpos=270,L={E}]{G}
\Vertex[x=0,y=0,Lpos=270,L={F}]{A}
\Vertex[x=10,y=0,Lpos=270,L={G}]{H}
\Vertex[x=20,y=0,Lpos=270,L={H}]{C}
\Vertex[x=-10,y=10,Lpos=90,L={A}]{E}
\Vertex[x=0,y=10,Lpos=90]{B}
\Vertex[x=10,y=10,Lpos=90,L={C}]{F}
\Vertex[x=20,y=10,Lpos=90]{D}
\AddVertexColor{black}{B,F,D,E}
\Edge[color=red](B)(G)
\Edge[color=red](D)(H)
\Edge[color=red](F)(C)
\Edge[color=red](A)(E)
\Edge[color=blue](D)(G)
\Edge[color=blue, style=Dash](B)(H)
\Edge[color=blue](E)(C)
\Edge[color=blue, style=Dash](A)(F)
\Edge(D)(C)
\Edge[style=Dash](F)(H)
\Edge[style=Dash](E)(G)
\Edge(A)(B)
\Edge[color=green](G)(F)
\Edge[color=green,style=Dash](H)(E)
\Edge[color=green](A)(D)
\Edge[color=green,style=Dash](C)(B)
\end{tikzpicture}
\end{center}
  \caption{Two examples of Adinkras.}
  \label{fig:Adinkras}
\end{figure}

\end{example}

We have the following theorem concerning the underlying graph of an Adinkra.

\begin{theorem} \cite{DFGHILM} \label{thm:underlying_G}
A graph admits an Adinkra structure if and only if it is the quotient of a hypercube $Q_t$ by a {\em doubly even code} (a subspace $\Ccal$ of $\FF_2^t$ where the size of the support of each element is divisible by $4$), which is necessarily a Cayley graph of $\FF_2^n$, where $n=t-\dim\Ccal$.
\end{theorem}

\subsection{Critical Groups and Their Monodromy Pairings}

\begin{definition}
The {\em Laplacian} of $G_{\sigma}$ is $L(G_\sigma)=D-A_{\sigma}$, where $D$ is the diagonal matrix whose entries are the vertex degrees, and $A_{\sigma}$ is the signed adjacency matrix in which a positive (respectively, negative) edge $xy$ is represented by $A_{xy}=A_{yx}=1$ (respectively, $-1$).
The Laplacian of an ordinary graph $G$ can be viewed/defined as $L(G_\sigma)$ with $\sigma\equiv +$.

The {\em critical group} $\K(G_\sigma)$ is the torsion part of the cokernel of $L(G_\sigma)$ over $\ZZ$.
In the case of Adinkras (indeed, any {\em unbalanced} signed graphs), the cokernel itself is finite, whereas it is always of rank 1 for ordinary graphs.
\end{definition}

As two simple observations: (1) switching vertices perserves $\K(G_\sigma)$, and whether the signed graph is an Adinkra; and (2) the rank of $L(G_\sigma)$ over $\FF_2$ ({\em $2$-rank}) is independent of $\sigma$, in particular, it is equal to that of $L(G)$, and the number of even invariant factors of $\K(G_\sigma)$ equals $v$ minus ($2$-rank of $L$ + rank of $\coker L(G_\sigma)$).

The critical group of an Adinkra (or more generally, the torsion of the cokernel of any symmetric integer matrix) is equipped with a canonical pairing $\langle \cdot,\cdot\rangle$ taking values in $\mathbb{Q}/\mathbb{Z}$, known as the monodromy pairing.
It is related to several other pairings in arithmetic geometry and discrete potential theory \cite{BS_Pairing, BL_Pairing}.

\begin{definition} \label{def:pairing}
Let ${\bf x},{\bf y}\in\mathbb{Z}^v$ be two vectors representing two elements of $\K(\A)$.
Choose a positive integer $m$ such that $L(\A){\bf f}=m{\bf x}$ for some ${\bf f}\in \mathbb{Z}^v$.
Then the {\em monodromy pairing} between $[{\bf x}],[{\bf y}]$ is $\langle [{\bf x}],[{\bf y}]\rangle:=\frac{{\bf f}^T{\bf y}}{m}\in\mathbb{Q}/\mathbb{Z}$.
\end{definition}

\begin{proposition} \cite[Lemma~1.1]{BL_Pairing} 
The pairing is well-defined, bilinear, and symmetric.
\end{proposition}

\section{Proof of the Main Theorem}

Index the rows and columns of $L(\A)$ by $V(\A)\cong\FF_2^n$, and denote by $\{{\bf e}_u: u\in V(\A)\}$ the standard basis of $\ZZ^{V(\A)}$.

We first collect some results on the Laplacians and critical groups of Adinkras from \cite{IKKY} that can be obtained in a more elementary manner.

\begin{theorem} \label{thm:Adinkra_basic}
Let $\A$ be an $N$-colored Adinkra on $v$ vertices.
Then $L(\A)$ has exactly two distinct eigenvalues $N\pm\sqrt{N}$ of equal multiplicities $v/2$, and $|\K(\A)|=\det(L(\A))=(N^2-N)^{v/2}$.

The invariant factors $f_1\mid f_2\mid\ldots\mid f_v$ of $L(\A)$ satisfy the relation $f_if_{v-i+1}=N^2-N,\forall i$.
Moreover, for $i>v/2$, $(N-1)\mid f_i$.
\end{theorem}

The key (and neat) observation is that the signed boundaries of a family of monochromatic edges are ``orthonormal'' with respect to $\langle \cdot,\cdot\rangle$.

\begin{proposition} \label{prop:pairing}
Let $\A$ be an Adinkra of $N\geq 3$ colors and let $uv$ be an edge of $\A$ of sign $\epsilon$.
Then $\langle {\bf e}_u-\epsilon {\bf e}_v, {\bf e}_u-\epsilon {\bf e}_v\rangle=\frac{2}{N}\neq 0\in\mathbb{Q}/\mathbb{Z}$.
Let $xy$ be another edge of the same color and of sign $\epsilon'$.
Then $\langle {\bf e}_u-\epsilon {\bf e}_v, {\bf e}_x-\epsilon' {\bf e}_y\rangle=0$.
\end{proposition}

\begin{proof}
Without loss of generality, we may assume $\epsilon=\epsilon'=+$ by switching.
By the first half of Theorem~\ref{thm:Adinkra_basic}, $L(\A)$ has two distinct eigenvalues, so it satisfies the condition in \cite[Equation~(3.3)]{DBS}, and we can solve $m({\bf e}_u-{\bf e}_v)=L_\A{\bf f}$ by
\begin{equation} \label{eq:L}
(N^2-N)({\bf e}_u-{\bf e}_v)
= L_\A[(N-1){\bf e}_u-(N-1){\bf e}_v+\sum_{w\in N(u)\setminus v}\sigma(uw){\bf e}_{w}-\sum_{w\in N(v)\setminus u}\sigma(vw){\bf e}_w],
\end{equation}
here $N(x)$ is the neighborhood of the vertex $x$ and $\sigma(e)$ is the sign of the edge $e$.

By Definition~\ref{def:pairing} and (\ref{eq:L}), $\langle {\bf e}_u-{\bf e}_v, {\bf e}_u-{\bf e}_v\rangle=\frac{2(N-1)}{N^2-N}=\frac{2}{N}$, which is non-zero in $\mathbb{Q}/\mathbb{Z}$ as $N\geq 3$.

For the second statement, $\{x,y\}$ and $\{u,v\}$ are necessarily disjoint by (2) of Definition~\ref{def:Adinkra}.
If $x$ is adjacent to $u$ along a positive edge of color $c$, (1) of Definition~\ref{def:Adinkra} guarantees that $x$ is not adjacent to $v$, and (3) and (4) of Definition~\ref{def:Adinkra} ensure $y$ is adjacent to $v$ along a negative edge of the same color but not to $u$.
Now (\ref{eq:L}) implies $\langle {\bf e}_u-{\bf e}_v, {\bf e}_x-{\bf e}_y\rangle=\frac{1-1}{N^2-N}=0$; the cases when $x$ is adjacent to $v$ and/or the edge is negative are essentially the same.
The case when $x$ is not adjacent to $u$ nor $v$ is easier as $y$ is not adjacent to $u,v$ either, and the pairing is simply $\frac{0-0}{N^2-N}=0$.
\end{proof}

Next, we state the result from \cite{DBS} that explains how an orthonormal subset implies a ``rectangular'' subgroup.

\begin{theorem} [{\cite[Theorem~4.5]{DBS}}]  \label{thm:tail_heavy}
Let $G$ be a finite abelian group equipped with a monodromy pairing $\langle \cdot,\cdot \rangle$.
Suppose there exist $g_1,\ldots,g_l\in G$ whose pairwise pairings are zero, and for every $i$, $\langle g_i,g_i \rangle = \frac{\mu}{\eta}$ for relatively prime $\mu,\eta\in\mathbb{Z}_{>0}$.
Then $G$ contains a subgroup isomorphic to $(\mathbb{Z}/\eta\mathbb{Z})^l$.
\end{theorem}

\begin{proof}[Proof of Theorem~\ref{mainthm}]
The only Adinkra with parameter $N=2$ is the $4$-cycle with 1 (or 3) negative edges, in which the theorem can be easily verified.

For $N\geq 3$, fix a color of the Adinkra.
By switching if necessary, we may assume the $v/2$ edges of that color are all positive.
Applying the calculation in Proposition~\ref{prop:pairing} to Theorem~\ref{thm:tail_heavy} with $\eta=N/\gcd(2,N)$, we know that $\K(\A)$ contains a subgroup isomorphic to $(\ZZ/\frac{N}{\gcd(2,N)}\ZZ)^{v/2}$.

Hence, by an elementary fact on invariant factors and subgroups (for reference, see \cite[Lemma~4.4]{DBS}), for every $i>v/2$, $\frac{N}{\gcd(2,N)}\mid f_i$.
Combining that $(N-1)\mid f_i, \forall i>v/2$, we have $\frac{N^2-N}{\gcd(2,N)}\mid f_i$, which in turn forces each $f_i$ with $i\leq v/2$ to be either $1$ or $2$ by the second half of Theorem~\ref{thm:Adinkra_basic}.
The number of even invariant factors (necessarily 2) in the first half is $m$, so the number of invariant factors in the second half that are equal to $\frac{N^2-N}{2}$ is also $m$, and the remaining non-trivial invariant factors must be $N^2-N$, the claimed structure of $\K(\A)$ follows.
\end{proof}

\section{Non-generic Adinkras}

Corollary~\ref{maincoro} is straightforward from the main theorem.

{\begin{proof}[Proof of Corollary~\ref{maincoro}]
From the aforementioned observation, the signature does not affect the $2$-rank of the Laplacian, which determines the critical group.
\end{proof}}

We recall some background of the corollary: while Theorem~\ref{thm:underlying_G} classifies which graphs admit an Adinkra structure, for a given such graph, there can be multiple (even up to natural notions of isomorphism) Adinkra structures.
In particular, while $\K(\A)$ is invariant under vertex switchings and {\em color-preserving graph automorphisms}, there can exist different Adinkra signatures on a graph $G=Q_t/\Ccal$ that are not equivalent by these two operations, hence the original conjecture is not a vacuous question.

Indeed, $G$ admits inequivalent signatures if and only if $\Ccal$ contains the all one codeword ${\bf 1}\in\FF_2^t$ \cite{DIL}, and in the language of Cayley graphs, if and only if the sum of generators is zero.
These Cayley graphs are {\em non-generic} in the sense of \cite{GMMY}, and they are precisely the Cayley graphs on $\FF_2^n$ whose $2$-rank drops below $2^{n-1}$, i.e., $m>0$ in the main theorem.
Therefore, the classes of Adinkras (or the underlying graphs thereof) that behave non-trivially in terms of signatures and critical groups turn out to be the same.

We use this opportunity to mention one more {\em possible} instance that the very class of Adinkras is special.
As referred to in the introduction, Theorem~\ref{thm:IKKY} was proven by deforming the critical group into a $\ZZ[x]$-module.
This could be done by considering the following matrix $\hat{L}(\A)$ over $\ZZ[x]$: fix an arbitrary color $c$,  replace the diagonal entries of $L(\A)$ by $x+(N-1)$, and replace the off-diagonal entries $\pm 1$ corresponding to edges of color $c$ by $\pm x$.
Since $\ZZ[x]$ is not a PID, it is not obvious that the SNF of $\hat{L}(\A)$ exists, and it was conjectured in \cite{IKKY} that the SNF exists if and only if $\K(\A)\cong(\ZZ/(N^2-N)\ZZ)^{v/2}$, which we now know the latter is true if and only if $\A$ is generic.
It was only stated in \cite{IKKY} as a fact without proof that the forward direction is true, so we fill in the argument below:

\begin{proposition}
When $\A$ is non-generic, the SNF of $\hat{L}(\A)$ does not exist.
\end{proposition}

\begin{proof}
By \cite[Corollary 29]{IKKY}, the determinant of $\hat{L}(\A)$ is equal to
$$(2(N-1)x+(N-1)(N-2))^{v/2}.$$
Since $\A$ is non-generic, $N$ must be an even number: e.g., if $\gcd(2,N)=1$, then $f_i=(N-1)N$ for all $i>v/2$ from the proof of the main theorem.

Suppose the SNF of $\hat{L}(\A)$ exists, and the invariant factors are $\hat{f}_1\mid\ldots\mid\hat{f}_v$.
Then whenever $2$ or $x+\frac{N-2}{2}$ divides $\hat{f}_i$ for some $i\leq v/2$, the same can be said for $\hat{f}_j$ with $j\geq i$, a contradiction to the fact that $\prod_{i=1}^v \hat{f}_i = \pm 2^{v/2}(N-1)^{v/2}(x+\frac{N-2}{2})^{v/2}$, and that $\ZZ[x]$ is a UFD with $2,x+\frac{N-2}{2}$ not dividing $N-1$.
Hence, $\hat{f}_i\mid (N-1)$ for $i\leq v/2$.

The SNF of $L(\A)$ can be obtained from the SNF of $\hat{L}(\A)$ by setting $x=1$ (see, for example, \cite[Lemma 27]{IKKY}), so the first half of the invariant factors of $L(\A)$ must be all odd, a contradiction.
\end{proof}

\section{Concluding Remarks}

In some sense, Theorem~\ref{mainthm} is the best possible result concerning $\K(\A)$ unless one is able to make progress on the $2$-rank of Cayley graphs, which is a non-trivial problem arguably orthogonal to the combinatorics of Adinkras\footnote{On the optimistic side, the author does not rule out the possibility of using Adinkras to approach problems in Cayley graphs.}.
However, as Adinkras are related to multiple mathematical topics \cite{Iga_Clifford, zhang}, putting our result in the context of those topics would be fruitful.

When studying the critical groups of Cayley graphs or many other graphs from algebra, the results and/or their proofs are often {\em directly} related to the algebraic origin of those graphs.
On the contrary, the works on critical groups of Adinkras so far mostly use the combinatorial axioms to develop alternative algebraic setups for the problem.
So it is interesting to interpret the result here directly using supersymmetry algebras.
For example, do supersymmetry algebras corresponding to non-generic Adinkras also special in some way?

On the geometric side, every Adinkra can be canonically embedded to a Riemann surface in the sense of Grothendieck's dessins d'enfants, and some properties of those Riemann surfaces are related to the properties of Adinkras in a deep manner \cite{DIKLM1,DIKLM2}.
Meanwhile, the theory of critical groups is a discrete/tropical analogue of the theory of divisors and Jacobians of algebraic curves \cite{MB_Spec}.
Comparing the two worlds via the embedding is another direction worth looking into.
For example, elements ${\bf e}_u-{\bf e}_v$'s considered in the proof of Proposition~\ref{prop:pairing} are now divisors on the Riemann surface, does the monodromy pairing on $\K(\A)$ relate to any notion there?

Finally, one can also ask if there are other families of graphs or signed graphs whose critical groups can be approached in a similar fashion.
More generally, can the structure of every critical group be certified by demonstrating an ``orthogonal basis'' with respect to $\langle\cdot,\cdot\rangle$?
If not, how much information can the method provide?

\section*{Acknowledgements}
The author was supported by the Trond Mohn Foundation project ``Algebraic and Topological Cycles in Complex and Tropical Geometries'' at the University of Oslo.
He also thanks Kevin Iga for reading an early draft.
\vspace{-2mm}

\bibliographystyle{plain}

\bibliography{SRG}

\end{document}